\documentclass[a4paper, final]{amsart}

\usepackage[latin1]{inputenc}
\usepackage[T1]{fontenc}
\usepackage[english,french]{babel}
\usepackage{amsmath,amssymb,amscd,latexsym}

\numberwithin{equation}{section}

\theoremstyle{plain}
\newtheorem{Thm}{Th\'eor\`eme}[section]

\newtheorem{theorem}[Thm]{Theorem}
\newtheorem{proposition}[Thm]{Proposition}

\newtheorem{lemma}[Thm]{Lemma}

\theoremstyle{definition}

\newtheorem{definition}[Thm]{Definition}

\theoremstyle{remark}

\newtheorem{example}[Thm]{Example}
\newtheorem{remark}[Thm]{Remark}

\newcommand{\PScal}[2]{\left\langle{#1}\big\vert {#2}\right\rangle}

\newcommand{\Hc}{{\mathcal H}}
\newcommand{\Bc}{{\mathcal B}}
\newcommand{\Gc}{\mathcal{G}}

\newcommand{\N}{\mathbb{N}}
\newcommand{\Z}{\mathbb{Z}}
\newcommand{\R}{\mathbb{R}}
\newcommand{\C}{\mathbb{C}}
\newcommand{\F}{\mathbb{F}}

\begin{document}
\selectlanguage{english}

\title{Fixed point properties in the space of marked groups}
\author{Yves Stalder}
\date{March 27, 2008}
\address{Laboratoire de Math\'ematiques\\ Universit\'e Blaise Pascal\\ Campus
Universitaire des C\'ezeaux\\ 63177 Aubi\`ere cedex\\ France}
\email{yves.stalder@math.univ-bpclermont.fr}


\begin{abstract}
 We explain, following Gromov, how to produce uniform isometric actions of 
 groups starting from
 isometric actions without fixed point, using common ultralimits techniques.
 This gives in particular a simple proof of a result by Shalom: Kazhdan's
 property (T) defines an open subset in the space of marked finitely generated
 groups.
\end{abstract}

\maketitle

\section{Introduction}

In this expository note, we are interested in groups whose actions on
some particular kind
of spaces always have (global) fixed points.
\begin{definition}
 Let $G$ be a (discrete) group. We say that $G$ has:
 \begin{itemize}
 \item Serre's \emph{Property (FH)}, if any isometric $G$-action on an affine
 Hilbert space has a fixed point \cite[Chap~4]{HV};
 \item Serre's \emph{Property (FA)}, if any $G$-action on a simplicial tree (by
 automorphisms and without inversion) has a fixed point \cite[Chap~I.6]{Ser};
 \item \emph{Property (F$\R$A)}, if if any isometric $G$-action on a complete
 $\R$-tree has a fixed point \cite[Chap~6.b]{HV}.
\end{itemize}
These definitions extend to topological groups: one has then to require the
actions to be continuous.
\end{definition}

Such properties give information about the structure of the group $G$. Serre
proved that a countable group has Property (FA) if and only if (i) it is
finitely generated, (ii) it has no infinite cyclic quotient, and (iii) it is
not an amalgam \cite[Thm~I.15]{Ser}. Among locally compact, second 
countable
groups, Guichardet and Delorme proved that Property (FH) is equivalent to
Kazhdan's Property (T) \cite{Gui,Del}. Kazhdan groups are known to be compactly
generated and to have a compact abelianization; see e.g. Chapter 1 in \cite{HV}.
It is known that Property (FH) implies Property (F$\R$A)\footnote{This was
noticed by several people; see \cite[Chap~6.b]{HV}.}, which itself obviously
implies Property (FA).

We are interested in the behavior of these properties in Grigorchuk's space of
marked (finitely generated) groups (see Section \ref{defs} for definition).
One main aim of this note is to provide a simple proof of the following result,
which implies that any finitely generated Kazhdan
group is a quotient of a finitely presented Kazhdan group:
\begin{theorem}[Shalom \cite{Sha}]\label{IntOuv_T}
 Property (FH) defines an open subset in the space of marked groups.
\end{theorem}

Rather than just
prove this result, our purpose is to indicate a general scheme, which gives
a common strategy for proving Shalom's result and Theorems \ref{Int_T_ppf} and
\ref{IntOuv_FRA} below.
\begin{theorem}[Korevaar-Schoen \cite{KoSch}, Shalom \cite{Sha}]\label{Int_T_ppf}
A finitely generated group $G$ has Property (FH) if and only if every isometric
$G$-action on a Hilbert space almost has fixed points.
\end{theorem}
\begin{theorem}[Culler-Morgan]\label{IntOuv_FRA}
 Property (F$\R$A) defines an open subset in the space of marked groups.
\end{theorem}
In fact, one deduces the latter Theorem from \cite[Thm~4.5]{CM} by
applying it to a free group $\F_n$.
On the other hand, it is an open problem whether Property (FA) defines
an open subset in the space of marked groups.

The general (simple) idea for the scheme is, starting
from actions without (global) fixed points, to pass to a ``limit'' of the
spaces to get uniform actions (that is, not almost having fixed points; see
Section~\ref{defs}). It follows in fact the same strategy as in 
\cite[Sect 3.8.B]{Gro03}\footnote{This was pointed out to me by Yves de 
Cornulier, after this paper was accepted for publication and posted on the 
arXiv. There also exist good unpublished notes \cite{Cor} on this topic.}, 
and Theorems \ref{IntOuv_T} -- \ref{IntOuv_FRA} are parcular cases
of Gromov's results. 

\begin{theorem}[Gromov]\label{IntUltrSansPPF}
 Let $(G_k,S_k)_{k\in\N}$ be a sequence of marked groups converging to $(G,S)$.
 If each group $G_k$ acts without fixed point on a (non-empty) complete metric
 space $(X_k,d_k)$, then $G$ acts uniformly  on some ultralimit of the spaces
 $(X_k,d_k)$.
\end{theorem}

The definition of ultralimits will be given in Section \ref{UltSt}. Note that
we allow to rescale the spaces before taking the limit (see Theorem 
\ref{UltrSansPPF} for a more precise statement). The idea
to take ``limits'' of metric spaces is not new, even for such purposes.
Asymptotic cones, introduced by Gromov in \cite{Gro81} and defined rigorously
in \cite{Gro93, DrWi} are a major particular case of ultralimits which is
very useful in the study of metric spaces and groups, see e.g. \cite{Dru} and
references therein. Ultralimits appear explicitely in \cite{KleLe,BriH}, for
instance. Finally, let us mention that Korevaar and Schoen \cite{KoSch}
introduced limits of CAT(0) spaces
(with another process) and proved in this context results of the same spirit
as Theorem \ref{IntUltrSansPPF}.

Section \ref{defs} gives the necessary preliminaries. In Section \ref{UltSt},
we recall what ultralimits are and prove Theorem \ref{IntUltrSansPPF}. Finally,
Section 3 is devoted to applications to Properties (FH) and (F$\R$A).

\subsection*{Acknowledgements.} I would like to thank Benjamin Delay for having
pointed out a mistake in the proof of Theorem \ref{UltrSansPPF} and for his
suggestions about a first version of this paper. Thanks are also due to Yves
de Cornulier for references \cite{Cor,Gro03,Laf}, to Vincent
Guirardel for a useful discussion about the results in this text, and to the
referee for her/his valuable suggestions. Finally, I
am particularly grateful to Nicolas Monod and Alain Valette for their very
valuable advices and hints, and for their remarks about previous versions of
the text.

\section{Preliminaries}\label{defs}

\subsection{Terminology about group actions.} In this note, every metric space
is assumed to be non-empty and all groups considered are discrete ones.
Let $(X,d)$ be a metric space and let $G$ be a
group acting on it by isometries. The action is said to \emph{almost
have fixed points} if, for all $\varepsilon>0$ and for all finite subset
$F\subseteq G$, there exists $x\in X$ such that $d(g\cdot x, x) < \varepsilon$
for all $g\in F$; it is said to be \emph{uniform} otherwise.

An action with a global fixed point almost has fixed points, but the converse
strongly does not hold. Indeed, the following examples show that an action with
almost fixed points can be \emph{metrically proper}, that is such that for any
$x\in X$ and $R>0$, the set $\{g\in G : d(x,gx) \leq R\}$ is finite.
\begin{example}[on the hyperbolic plane]\label{Prop_ppf_hyp}
 Let $\mathbb H^2$ be the
 Poincar\'e upper half-plane and define $\varphi: \mathbb H^2 \to \mathbb H^2$
 by $\varphi(z) = z+1$. This gives a $\Z$-action by isometries on
 $\mathbb H^2$ which is metrically proper and almost has fixed points.
\end{example}
\begin{example}[on a Hilbert space]\label{Prop_ppf}
 Let $S$ be the shift operator on $\ell^2\Z$ (defined by $(S\xi)(n) =
 \xi(n-1)$) and $\delta_0 \in \ell^2\Z$ the Dirac mass at $0$. The affine map
 $\xi \mapsto S\xi + \delta_0$ defines a $\Z$-action by isometries on
 $\ell^2\Z$ which is metrically proper and almost has fixed points.
\end{example}
\begin{remark}\label{ppf_tf}
 In the case of finitely generated groups, the definition can be made easier:
 if $S$ is a finite generating set of $G$, a $G$-action by isometries on $X$
 almost has fixed points if and only if, for all $\varepsilon>0$, there exists
 $x\in X$ such that $\max\{d(x,s\cdot x) : s\in S\} < \varepsilon$.
\end{remark}

\subsection{The space of marked groups.} Let us recall that a \emph{marked
group on $n$ generators} is a pair $(G,S)$ where $G$ is a group and $S\in G^n$
generates $G$. A marked group $(G,S)$ defines canonically a quotient $\phi:\F_n
\to G$, and vice-versa. Moreover, for such a quotient, we may consider the
normal subgroup $N = \ker(\phi) \triangleleft \F_n$. Two marked groups, or two
quotients, are said to be equivalent if they define the same normal subgroup of
$\F_n$. Abusing terminology, we denote by $\Gc_n$ the set of (equivalence
classes of) marked groups on $n$ generators, or the set of (equivalence classes
of) quotients of $\F_n$, or the set of normal subgroups of $\F_n$.

We now describe Grigorchuk's topology on $\Gc_n$ \cite{Gri84}, which
corresponds to an earlier construction by Chabauty \cite{Chab}; for
introductory expositions, see \cite{Cha, ChaGui, Pau}. Denote by $B_r$ the ball
of radius $r$ in $\F_n$ (centered at the trivial element). Given normal
subgroups $N \neq N' \triangleleft \F_n$, we set
\[
 d(N,N') := \exp\left(-\max\big\{ r\in \N : \, N' \cap B_r = N \cap B_r \big\}
 \right) \ .
\]
This turns $\Gc_n$ into a compact, ultrametric, separable space. The map
\[
\big(G,(s_1, \ldots s_n)\big) \mapsto \big(G,(s_1, \ldots s_n,1_G)\big)
\]
defines an isometric embedding $\Gc_n \to \Gc_{n+1}$ for all $n\in\N$. We
denote by $\Gc$ the direct limit of this directed system of topological spaces
and call $\Gc$ the \emph{space of (finitely generated) marked groups}. Note
that $\Gc_n$ is open and closed in $\Gc$ for all $n$. Given
$N\triangleleft \F_n$ and $N_k \triangleleft \F_n$ for $k\in \N$, one has
$\lim N_k = N$ if and only if, for all  $g\in \F_n$: ($g\in N
\Longleftrightarrow g\in N_k$ for $k$ sufficiently large) and ($g\notin N
\Longleftrightarrow g\notin N_k$ for $k$ sufficiently large).

\section{From actions without fixed point to uniform actions}\label{UltSt}

\subsection{Being far from almost fixed points.} Let $(G,S)$ be a marked
(finitely generated) group. Let it act by isometries on a  metric
space $(X,d)$ and set $\delta(x) = \max\{d(x,sx): s\in S\}$ for any $x\in X$. A
point $x\in X$ is fixed by $G$ if and only if $\delta(x)=0$, and the action
almost has fixed points if and only if $\inf\{\delta(x): x\in X\} = 0$.

\begin{remark}\label{SerreLemma}
 As $G$ is finitely generated, a $G$-action on an $\R$-tree $T$ almost has
 fixed points if and only if it has
 a fixed point. Indeed, if the action on $T$ has no fixed point, there exists
 $g\in G$  which induces a hyperbolic isometry of $T$ --- this can be deduced from
 \cite[Cor~2.3]{Tignol}; see also  \cite[Prop~II.2.15]{MoSha} or
 \cite[Exercise 2.8]{Best}. Thus, the $G$-action on $T$ is uniform.
\end{remark}
This property is specific to $\R$-trees, as illustrated in Examples
\ref{Prop_ppf_hyp} and \ref{Prop_ppf}. For general metric spaces, this Section
will explain how to produce uniform actions from actions without fixed points.

\begin{lemma}\label{Lip}
 We have $|\delta(x) - \delta(y)| \leq 2d(x,y)$ for all $x,y\in X$.
\end{lemma}
\begin{proof}
 For any $s\in S$, the triangle inequality gives
 \[
  d(x,sx) \leq d(x,y) + d(y,sy) + d(sy,sx)  \leq d(y,sy) + 2 d(x,y) \ ,
 \]
 which implies $\delta(x) \leq \delta(y) + 2 d(x,y)$. We then deduce similarly
 $\delta(y) \leq \delta(x) + 2 d(x,y)$.
\end{proof}

We now introduce one key ingredient to produce uniform actions, which has been
directly inspired from Lemma 6.3 in \cite{Sha}; see also Proposition 4.1.1 in
\cite{KoSch} and (the proof of) Lemma 2 in \cite{Laf}. It asserts, that we may 
find points which are, roughly speaking,
far from almost fixed points.

\begin{lemma}\label{keyLemma}
 Assume the space $X$ is complete and the action has no fixed point. Then, for
 all $n\in\N^*$, there exists $x_n\in X$ such that:
 \[
  \text{for all } y\in X, \quad d(y,x_n) \leq n\delta(x_n) \ \Longrightarrow
  \ \delta(y) \geq \frac{\delta(x_n)}{2} \ .
 \]
\end{lemma}
\begin{proof}
 Let us assume by contradiction that there exists some $n\in\N^*$ such that,
 for all $x\in X$, there exists $y=y(x) \in X$ which satisfies both $d(y,x)
 \leq n\delta(x)$ and $\delta(y) < \delta(x)/2$.
 Let us now take some $z_0 \in X$ (recall that $X$ is non-empty). Then, we
 define inductively a sequence of points $z_k$ such that $d(z_{k+1},z_k)
 \leq n\delta(z_k)$ and $\delta(z_{k+1}) < \delta(z_k)/2$ for all $k\in\N$.
 Consequently, we have $\delta(z_k) < \delta(z_0)/2^k$, whence $d(z_{k+1},z_k)
 \leq n\delta(z_0)/2^k$. Since $X$ is complete, this shows that $z_k$ converges
 to some point $z$ as $k$ tends to $\infty$. Thanks to Lemma \ref{Lip}, we
 obtain $\delta(z) = \lim_{k\to\infty} \delta(z_k) = 0$. Hence, $z$ is a fixed
 point of the $G$-action, a contradiction.
\end{proof}

Note that the hypotheses on $X$ made in Lemma \ref{keyLemma} cannot be dropped.
Indeed, to obtain counter-examples, consider $\Z$ acting by rotations on $\C$,
respectively $\C \setminus\{0\}$.

\subsection{Ultrafilters and ultralimits of metric spaces.} Bourbaki defines
ultrafilters to be maximal filters \cite{Boutopo}. However, we think
slightly differently to ultrafilters in this note.
\begin{definition}
 An \emph{ultrafilter} on some (non-empty) set $E$ is a finitely-additive,
 $\{0,1\}$-valued measure on $\mathcal P(E)$, that is a function
 $
 \omega : \mathcal P(E) \to \{0,1\}
 $
 which satisfies: (i) $\omega(A\cup B) = \omega(A) + \omega(B)$ whenever $A\cap
 B = \emptyset$, and (ii) $\omega(E) = 1$
\end{definition}

Note that (i) and (ii) imply $\omega(\emptyset) = 0$. In this note, we shall
only need to consider ultrafilters
on $\N$. The following well-known Lemma establishes the equivalence with
Bourbaki's definition. Its proof is given for
completeness.
\begin{lemma}
 A function $\omega: \mathcal P(E) \to \{0,1\}$ is an ultrafilter if and only
 if $\mathcal F = \{A\in\mathcal P (E) : \omega(A) = 1\}$ satisfies:
\begin{enumerate}
 \item $\emptyset \notin \mathcal F$ and $E \in \mathcal F$;
 \item if $A\in \mathcal F$ and $A \subseteq B \subseteq E$, then
 $B \in \mathcal F$;
 \item if $A,B \in  \mathcal F$, then $A \cap B \in \mathcal F$;
 \item for any $A \subseteq E$, one has $A \in \mathcal F$ or
 $E\setminus A \in \mathcal F$.
\end{enumerate}
\end{lemma}

\begin{remark}
 Properties (1)--(3) are precisely the axioms of \emph{filters} in \cite{Boutopo}.
\end{remark}

\begin{proof}
 Suppose first $\omega$ is an ultrafilter. Then (1), (2) and (4) are obvious.
 If $A,B \in  \mathcal F$, then (2) gives $\omega(A\cup B) = 1$ and (i) implies
 $\omega((A\cup B)\setminus B) = 0$, $\omega((A\cup B)\setminus A) = 0$ and
 finally
 \[
  \omega(A\cap B) = \omega(A\cup B) - \omega((A\cup B)\setminus B) -
  \omega((A\cup B)\setminus A) = 1 \ .
 \]
 Hence (3) is proved.

 Conversly, we now assume (1)--(4). Then (ii) is obvious. Properties (1) et (3)
 imply: (5) for any $C,D\subseteq E$ such that $C\cap D = \emptyset$, one has
 $C\notin \mathcal F$ or $D\notin \mathcal F$. To prove (i), let us now take
 $A,B\subseteq E$ such that $A\cap B = \emptyset$. The case $\omega(A) = 1 =
 \omega(B)$ is impossible by (5). In case $\omega(A) = 1$ and $\omega(B)=0$
 (or $\omega(A) = 0$ and $\omega(B)=1$), property (2) implies $\omega(A\cup B)
 = 1$. Finally, if $\omega(A) = 0 = \omega(B)$, (4) gives $E\setminus
 A\in \mathcal F$ and $E\setminus B\in \mathcal F$. We deduce then
 $(E\setminus A)\cap(E\setminus B)\in \mathcal F$ by (3) and $\omega(A\cup B)
 = 0$ by (5). Hence (i) is proved.
\end{proof}

\begin{example}
 Given any $a\in E$, there is an ultrafilter $\delta_a$ defined by
 $\delta_a(A)= 1$ if $a\in A$ and $\delta_a(A)=0$ otherwise. Such an
 ultrafilter is called \emph{principal}.
\end{example}
\begin{definition}
 Let $\omega$ be an ultrafilter on $\N$ and let $X$ be a metric space. A
 sequence $(x_n)$ in $X$ is said to converge to $x\in X$ relative to $\omega$
 if, for any neighborhood $V$ of $x$, the set $\{n\in\N : x_n \in V\}$ has
 $\omega$-measure $1$.
\end{definition}
The limit of a sequence, provided it exists, is unique, and we write
$\lim_{n\to \omega} x_n = x$, or $\lim_{\omega} x_n = x$. From this point of
view, the interesting ultrafilters are the non-principal ones: for instance,
any sequence $(x_n)$ converges to $x_k$, relative to the principal ultrafilter
$\delta_k$. The existence of non-principal ultrafilters follows from
Zorn's Lemma (see e.g \cite{Boutopo}, or Exercise I.5.48 in \cite{BriH}). We
shall use the following well-known result implicitly in the text. It ensures
 that any bounded sequence of
 real (or complex) numbers is $\omega$-convergent.
\begin{proposition}[\cite{Boutopo}]
 If $X$ is a compact metric space and $\omega$ is an ultrafilter on $\N$, then
 any sequence $(x_n)$ in $X$ is $\omega$-convergent.
\end{proposition}

We now define ultralimits of metric spaces, essentially as in \cite{KleLe} or
\cite{BriH}, (except that we
allow to ``rescale'' the spaces before to take the limit, as is done in the
construction of asymptotic cones, for instance). Let us consider sequences
$(X_k,d_k,*_k)_{k\in\N}$ of pointed metric spaces and $r=(r_k)$ of positive
numbers. Set
\[
 \Bc_{r} = \left\{x \in \prod_{k\in\N} X_k : \text{ the sequence }
 (r_k d_k(x_k,*_k))_k \text{ is bounded}\right\} \ .
\]
If some group $G$ acts by isometries on the spaces $X_k$, its diagonal action
may not stabilize $\Bc_{r}$. A necessary and sufficient condition is that the
sequence $(r_k d_k(g\cdot *_k,*_k))$ is bounded for any generator $g$ of $G$.
If this condition is fulfilled, we say that $G$ \emph{acts diagonally} on
$\Bc_{r}$.
For any ultrafilter $\omega$ on $\N$, we may endow $\Bc_{r}$ with the
pseudo-distance $d_{\omega,r}(x,y) = \lim_\omega r_k d_k(x_k,y_k)$. If $G$ acts
diagonally on $\mathcal B_r$, the diagonal action is isometric.
\begin{definition}
 Let $\omega$ be some non-principal ultrafilter on $\N$. The \emph{ultralimit}
 (relative to scaling factors $(r_k)$ and to $\omega$) of the sequence
 $(X_k,d_k,*_k)$ is the pointed metric space $(X_{\omega,r},d_{\omega,r},
 *_{\omega,r})$, where $X_{\omega,r}$ is the separation of
 $(\Bc_{r},d_{\omega,r})$ and $*_{\omega,r}$ denotes either the point
 $(*_k)_k \in \mathcal B_r$, or its image in $X_{\omega,r}$.
\end{definition}

Note that if $G$ acts diagonally on $\mathcal B_r$, the
diagonal action induces an isometric action on every ultralimit $X_{\omega,r}$,
which we call again diagonal.
If the sequence $(X_k,d_k)_k$ is constant and if $r_k \to 0$, one gets the
notion of asymptotic cone, due to Gromov \cite{Gro81,Gro93}, and van den Dries
and Wilkie \cite{DrWi}.

\begin{proposition}[Gromov \cite{Gro03}]\label{pf_diag}
 Let $G$ be a finitely generated group acting by isometries on complete metric
 spaces $(X_k,d_k)$ for $k\in\N$. If these actions have no fixed point, then
 there exist scaling factors $r_k > 0$ and base points $*_k \in X_k$ such that:
 \begin{enumerate}
  \item the group $G$ acts diagonally on $\mathcal B_r$;
  \item for any non-principal ultrafilter $\omega$ on $\N$, the diagonal
  action of $G$ on $X_{\omega,r}$ is uniform.
 \end{enumerate}
\end{proposition}

Korevaar and Schoen \cite{KoSch} used the same idea to rescale spaces, and then
take a limit, to produce uniform actions from actions without fixed points on
CAT(0) spaces. On the other hand their construction of limits uses properties
of CAT(0) spaces.

\begin{proof}
 Let $S$ be a finite generating set of $G$ and set $\delta_k(x) =
 \max\{d(x,sx): s\in S\}$ for all $x\in X_k$. By Lemma \ref{keyLemma}, we
 obtain points $*_k \in X_k$ for all $k\in \N$ such that:
\[
  \text{for all } y_k\in X_k, \quad d(y_k,*_k) \leq k\delta_k(*_k) \
  \Longrightarrow \ \delta_k(y_k) \geq \frac{\delta_k(*_k)}{2} \ .
 \]
 We now set $r_k = \delta_k(*_k)^{-1}$, which are well-defined since the
 actions have no fixed point. Thus, we have $r_k d_k(*_k,s\cdot *_k) \leq 1$
 for all $k$, so that (1) is satisfied.

 To prove (2), we consider some non-principal ultrafilter $\omega$ on $\N$.
 For any $y=(y_k)\in \mathcal B_r$, the sequence $(d(y_k,*_k)/\delta_k(*_k))_k$
 is bounded. Hence, for $k$ sufficiently large, one has
 $d(y_k,*_k) \leq k\delta_k(*_k)$, which implies
 $\delta_k(y_k) \geq \delta_k(*_k)/2$.

 For all $s\in S$, set now $A_s = \{k\in \N : r_kd_k(y_k,sy_k) \geq 1/2\}$.
 The former argument implies $k\in\bigcup_{s\in S} A_s$ for $k$ large enough,
 whence $\omega\big(\bigcup_{s\in S} A_s\big) = 1$. Since $S$ is finite, there
 exists $s\in S$ such that $\omega(A_s)=1$, which shows that
 $d_{\omega,r}(y,sy) = \lim_\omega r_kd_k(y_k,sy_k) \geq 1/2$.
\end{proof}

Let us now make Theorem \ref{IntUltrSansPPF} precise.

\begin{theorem}[Gromov \cite{Gro03}]\label{UltrSansPPF}
 Let $(G_k,S_k)_{k\in\N}$ be a sequence of marked groups converging to $(G,S)$
 in the space $\mathcal G_n$. If each group $G_k$ acts without fixed point on
 a complete metric space $(X_k,d_k)$, then there exists scaling factors
 $r_k > 0$ and base points $*_k \in X_k$ such that:
 \begin{enumerate}
  \item the free group $\F_n$ acts diagonally on $\mathcal B_r$;
  \item for any non-principal ultrafilter $\omega$ on $\N$, the diagonal action
  of $\F_n$ on the ultralimit $X_{\omega,r}$ is uniform;
  \item the diagonal action factors through the epimorphism $\F_n \to G$
  associated to $(G,S)$.
 \end{enumerate}
\end{theorem}

\begin{proof}
 Let $N_k$ and $N$ be the normal subgroups of $F_n$ associated with $(G_k,S_k)$
 and $(G,S)$ respectively. The $G_k$-actions on the spaces $X_k$ give
 $\F_n$-actions which are trivial on $N_k$. Proposition \ref{pf_diag} gives then
 scaling factors $r_k > 0$ and points $*_k \in X_k$ such that conditions (1)
 and (2) are fulfilled.

 To prove (3), it suffices to see that the diagonal action is trivial on $N$.
 Let us take $g\in N$. As $N_k \to N$, we have $g\in N_k$ for $k$ sufficiently
 large. Take now $y=(y_k) \in \Bc_{\omega,r}$. Since $g\cdot y_k = y_k$ for $k$
 large enough, we get $\lim_\omega r_k d_k(y_k,g y_k) = 0$, whence
 $d_{\omega,r}(y,g y) = 0$. The subgroup $N$ acts trivially on $X_{\omega,r}$,
 as desired.
\end{proof}


\section{Applications to fixed point properties}\label{Thms}
In this Section, we apply Theorem \ref{IntUltrSansPPF} (or Theorem
\ref{UltrSansPPF}) to obtain results about fixed point properties on groups.
Another ingredient is to identify classes of metric spaces which are stable by
ultralimits. Let us record two easy observations:

\begin{remark}\label{act_compl}
 When a group acts isometrically on a metric space, the action can be extended
 to the completion. Moreover, if the action is uniform, then so is the
 extension.
\end{remark}
\begin{remark}\label{ouv_n}
 The subsets $\Gc_n$ being an open cover of $\Gc$, a property defines an open
 set in $\Gc$ if and only if it defines an open set in every $\Gc_n$.
\end{remark}

\subsection{Fixed points in (affine) Hilbert spaces.}
We consider (affine) Hilbert spaces over $\R$, by forgetting the complex
structure and replacing the inner-product by its real part, if necessary.
For such spaces, it is
well-known that any isometry is an affine map. In Example \ref{Prop_ppf}, we
exhibited an isometric action on a Hilbert space without fixed point, which
almost has fixed points. On the other hand, Theorem \ref{IntUltrSansPPF} allows
to unify the proof of Theorems \ref{IntOuv_T} and \ref{Int_T_ppf}, that we now
recall.
\begin{theorem}[Shalom \cite{Sha}]\label{Ouv_T}
 Property (FH) defines an open subset in $\Gc$.
\end{theorem}
\begin{theorem}[Korevaar-Schoen \cite{KoSch}, Shalom \cite{Sha}]\label{T_ppf}
A finitely generated group $G$ has Property (FH) if and only if every isometric
$G$-action on a Hilbert space almost has fixed points.
\end{theorem}
Korevaar and Schoen also prove the following: if $\Gamma$ is a finitely
generated group with Property (FH) and if $X$ is a geodesically complete
CAT(0) space with curvature bounded from below, then any non-uniform
isometric action of $\Gamma$ on $X$ has a fixed point. To do this, they show
that some limits (in their sense) of geodesically complete CAT(0) space with
curvature bounded from below are Hilbert spaces (compare with Lemma
\ref{Hilb} below).

Let us now mention a result by Mok in the same vein \cite{Mok}: if $M$ is a
compact riemannian manifold and if $G = \pi_1(M)$, then $G$ has property (T) if
and only if, for any irreducible unitary representation $\pi$ of $G$, there is
no non-zero $E_\pi$-valued harmonic 1-form (where $E_\pi$ is the locally
constant Hilbert bundle on $M$ induced from $\pi$).
As this text was almost finished, we saw in the Appendix of \cite{Kle} a ``weak
version of some results in \cite{FiMa}'', which implies Theorem \ref{T_ppf}.
The proof in \cite{Kle} uses ultralimits of Hilbert spaces in a very similar
way as in this note (with less details).

For proofs of Theorems \ref{Ouv_T} and \ref{T_ppf}, we use the following
easy observation.
\begin{lemma}\label{Hilb}
 Any ultralimit of affine Hilbert spaces is an affine Hilbert space.
\end{lemma}
\begin{proof}
 Let us consider a sequence $(\Hc_k,d_k, *_k)_{k\in\N}$ of (pointed) affine
 Hilbert spaces, scaling factors $(r_k)_{k\in\N}$ and some non-principal
 ultrafilter $\omega$ on $\N$. We denote by $\Hc_k^0$ the (Hilbert) vector
 space under $\Hc_k$. Let $(\Hc_{\omega,r}^0,d_{\omega,r},0_{\omega,r})$ be
 the ultralimit of the sequence $(\Hc_k^0,||.||_k, 0)$, relative to $(r_k)$
 and to $\omega$. Then, $\Hc_{\omega,r}^0$ is a vector space with respect to
 operations  $(u_k) + (v_k)  :=  (u_k + v_k)$ and $\lambda \cdot (v_k) :=
 (\lambda\cdot v_k)$, and the formula $\PScal{u}{v} := \lim_\omega
 \PScal{r_k u_k}{r_k v_k}$ defines a bilinear form on $\Hc_{\omega,r}^0$.
 Moreover, for all $u,v\in \Hc_{\omega,r}^0$, we have:
 \[
  d_{\omega,r}(u,v)^2 = \lim_{k\to\omega} d_k(r_k u_k, r_k v_k)^2 =
  \lim_{k\to\omega} \PScal{r_k (u_k - v_k)}{r_k (u_k - v_k)} =
  \PScal{u-v}{u-v} \ ,
 \]
 so that $\Hc_{\omega,r}^0$ is an inner-product space. Since any ultralimit
 of metric spaces is complete\footnote{See e.g. Lemma 2.4.2 in \cite{KleLe} or Lemma I.5.53 in
 \cite{BriH}}, $\Hc_{\omega,r}^0$ is a Hilbert space.

 Finally, we consider the ultralimit $(\Hc_{\omega,r},d_{\omega,r},
 *_{\omega,r})$ of the sequence $(\Hc_k,d_k, *_k)$, relative to $(r_k)$ and
 to $\omega$. The action $(u_k)+(x_k) := (u_k+x_k)$ turns it into an affine
 space over $\Hc_{\omega,r}^0$. Hence it is an affine Hilbert space.
\end{proof}

\begin{proof}[Proof of Theorems \ref{Ouv_T} and \ref{T_ppf}]
 Let $(G_k)$ be a sequence in $\Gc$ which converges to some $G\in\Gc$. As the
 subspaces $\Gc_n$ form an open cover of $\Gc$, we find $n\in\N^*$ such that
 $G,G_k\in\Gc_n$. Assuming that every $G_k$ admits an isometric action without
 fixed point on some affine Hilbert space $\Hc_k$, Theorem \ref{IntUltrSansPPF}
 gives then an ultralimit $\Hc = \Hc_{\omega,r}$ on which $G$ acts uniformly.
 Moreover, $\Hc$ is an affine Hilbert space by Lemma~\ref{Hilb}.

 This proves Theorem \ref{Ouv_T}. Moreover, if we specialize to the case
 $G_k=G_0$ and $\Hc_k = \Hc_0$ for all $k$, it also proves the non-trivial
 part of Theorem \ref{T_ppf}.
\end{proof}

\subsection{Fixed points in (complete) $\R$-trees.}

Let us recall that $\R$-trees have been invented by Tits \cite{Tits}
\footnote{Unlike the definition we follow, Tits required $\R$-trees to be 
complete.}. Let us also recall (see e.g. Lemmata 1.2.6 and 2.4.3 in 
\cite{Chi}) that a metric
space $(X,d)$ is an \emph{$\R$-tree} if and only if the following conditions
both hold:
\begin{enumerate}
 \item it is \emph{geodesic}, that is, for any $x,y \in X$, there exists a
 map $c:[0, \ell]\to X$ such that $c(0) = x$, $c(\ell) = y$ and
 $d(c(t),c(t')) = |t-t'|$ for all $t,t' \in [0, \ell]$;
 \item it is \emph{$0$-hyperbolic}, that is, for all $x,y,z,t \in X$:
 \[
  d(x,y)+d(z,t) \leq \max\{d(x,z) + d(y,t) , d(y,z) + d(x,t)  \} \ .
 \]
\end{enumerate}
A map $c$ as in point (1) is called a \emph{geodesic} from $x$ to $y$. In what
follows, we shall need the fact that any ultralimit of $\R$-trees is an
$\R$-tree. One may argue by saying that the ultralimit is a quotient of some
subtree of a $\Lambda$-tree (where $\Lambda$ is the ultrapower of $\R$ with
respect to $\omega$), in which we identify points at infinitesimal distance.
However, we prefer a more pedestrian way.
\begin{lemma}\label{GeodHyp}
\begin{enumerate}
 \item Any ultralimit of geodesic spaces is a geodesic space;
 \item Any ultralimit of $0$-hyperbolic spaces is a $0$-hyperbolic space.
\end{enumerate}
\end{lemma}

This Lemma is easy. For instance, part (1) is an exercise in \cite{BriH}. We
nevertheless give a proof for completeness.

\begin{proof}
Let us consider some non-principal ultrafilter $\omega$, pointed metric spaces
$(X_k,d_k,*_k)$, and scaling factors $r_k$, for $k\in\N$.

 (1) Assume the spaces $X_k$ are geodesic and take $x\neq y\in X_{\omega,r}$,
 represented by elements $(x_k), (y_k)$ in $\mathcal B_r$. We set
 $\ell_k = d_k(x_k,y_k)$ and consider geodesics $c_k:[0,\ell_k] \to X_k$ from
 $x_k$ to $y_k$ (note that we may assume $x_k\neq y_k$ for all $k$). Setting
 $\ell = \lim_\omega r_k\ell_k = d_{\omega,r}(x,y)$, we define a map
 \[
  c \ : \ [0, \ell] \to \mathcal B_r \ ; \ t \mapsto
  \left(c_k\left( t \frac{\ell_k}{\ell} \right)\right)_{k\in\N}
 \]
 Then $d(c(t),c(t')) = \lim_\omega r_k d(c_k(t\ell_k/\ell), c_k(t'\ell_k/\ell))
 = \lim_\omega r_k(\ell_k/\ell)\cdot |t-t'| = |t-t'|$ for all $t,t'\in[0,\ell]$.
 Hence $c$ is a geodesic from $x$ to $y$.

 (2) Assume the spaces $X_k$ are $0$-hyperbolic and take $x,y,z,t\in
 X_{\omega,r}$, which are represented by elements $(x_k), (y_k), (z_k), (t_k)$
 in $\mathcal B_r$. Fixing $\varepsilon>0$, we have
 \[
  \omega\left(\left\{ k\in\N :
  \begin{array}{ccc}
    r_k d_k(x_k,z_k) \leq d_{\omega,r}(x,z) + \frac \varepsilon 2 & , &
    r_k d_k(y_k,t_k) \leq d_{\omega,r}(y,t) + \frac \varepsilon 2 \\
    r_k d_k(y_k,z_k) \leq d_{\omega,r}(y,z) + \frac \varepsilon 2 & , &
    r_k d_k(x_k,t_k) \leq d_{\omega,r}(x,t) + \frac \varepsilon 2
  \end{array}
  \right\}\right) = 1 \ .
 \]
 We now use $0$-hyperbolicity of the spaces $X_k$, which shows that the set
 \[
  \Big\{ k\in\N : \
  r_k d_k(x_k,y_k)+ r_k d_k(z_k,t_k) \leq \max\big\{d_{\omega,r}(x,z) +
  d_{\omega,r}(y,t) + \varepsilon , d_{\omega,r}(y,z) + d_{\omega,r}(x,t)
  + \varepsilon \big \}
  \Big\}
 \]
 has $\omega$-measure $1$.
 Hence, we obtain
 \[
 d_{\omega,r}(x,y)+d_{\omega,r}(z,t) \leq
 \max\big\{d_{\omega,r}(x,z) + d_{\omega,r}(y,t) , d_{\omega,r}(y,z) +
 d_{\omega,r}(x,t)  \big\} + \varepsilon \ .
 \]
 As $\varepsilon$ is arbitrary, this
 shows that the ultralimit $X_{\omega,r}$ is $0$-hyperbolic.
\end{proof}

We now recall and prove Theorem \ref{IntOuv_FRA} of the Introduction.

\begin{theorem}[Culler-Morgan \cite{CM}]\label{Ouv_FRA}
 Property (F$\R$A) defines an open subset in $\Gc$.
\end{theorem}
\begin{remark}\label{FRA_uncomplete}
 Let $G$ be a \textit{finitely generated} group. Then, $G$ has property (F$\R$A)
 if and only if every $G$-action on an $\R$-tree has a fixed point.
\end{remark}
\begin{proof}
 The completion of an $\R$-tree is an $\R$-tree \cite{Imrich} --- see also
 \cite[Cor~II.1.10]{MoSha} or
 \cite[Thm~2.4.14]{Chi}. Hence, we are done by Remarks \ref{SerreLemma}
 and \ref{act_compl}.
\end{proof}

\begin{proof}[Proof of Theorem \ref{Ouv_FRA}]
 Let $(G_k)$ be a sequence in $\Gc_n$ which converges to some point
 $G\in\Gc_n$, and such that any $G_k$ acts without fixed point on some
 complete $\R$-tree $T_k$. By remark \ref{ouv_n}, it suffices to show that
 $G$ acts without fixed point on some complete $\R$-tree.

 Theorem \ref{IntUltrSansPPF} gives an ultralimit $T=T_{\omega,r}$ of the
 $\R$-trees $T_k$ on which $G$ acts uniformly, and $T$ is an $\R$-tree by
 Lemma \ref{GeodHyp}. Hence, Remark \ref{FRA_uncomplete} concludes the
 proof.\footnote{In fact, use of Remark \ref{FRA_uncomplete} is superfluous, as
 it is known that every ultralimit of metric spaces is complete: see e.g.
 Lemma 2.4.2 in \cite{KleLe} or
 Lemma I.5.53 in \cite{BriH}}
\end{proof}
\begin{remark}
 The last proof does not work for simplicial trees: we used the fact that the
 class of $\R$-trees is closed under ultralimits.
\end{remark}
 In fact, as the referee pointed out, using Theorem \ref{IntUltrSansPPF} to 
 prove Theorem \ref{Ouv_FRA}
 is a little awkward: as Remark \ref{SerreLemma} tells us immediately that the
 $G_k$-actions on the $\R$-trees $T_k$ are uniform, it is unnecessary to make a
 clever choice of base points as in Proposition \ref{pf_diag} and Theorem
 \ref{UltrSansPPF}.

\bibliographystyle{alpha}
\def\cprime{$'$}

\end{document}